\documentclass[12pt]{article}
\usepackage{amssymb, amsfonts, amsmath}
\usepackage{hyperref}
\usepackage{xcolor}
\hypersetup{
	colorlinks,
	linkcolor={red!60!black},
	citecolor={green!60!black}
}
\newcommand{\qed}{\hskip 5mm \rule{2.5mm}{2.5mm}\vskip 10pt}

\newcommand{\N}{{\mathbb N}}

\newcommand{\proof}{\noindent{\em Proof:\ }}
\newcommand{\abs}[1]{ \left| #1 \right|}

\newcommand{\func}[2]{#1 \left( #2 \right)}

\newcommand{\parent}[1]{ \left( #1 \right)}

\newcommand{\setbuilder}[2]{ \left\{ #1 \mid #2 \right\}}

\numberwithin{equation}{section}
\newcommand{\range}[1]{\mathcal{R} \left( #1 \right)}

\newcommand{\dedotimes}[2]{\overline{#1 \overline{\otimes} #2}^{\delta}}

\newcommand{\dedclosure}[1]{\overline{#1}^{\delta}}

\begin{document}
\newtheorem{theorem}{Theorem}[section]
\newtheorem{definition}[theorem]{Definition}
\newtheorem{lemma}[theorem]{Lemma}
\newtheorem{note}[theorem]{Note}
\newtheorem{corollary}[theorem]{Corollary}
\newtheorem{proposition}[theorem]{Proposition}
\newtheorem{remark}[theorem]{Remark}
\renewcommand{\theequation}{\arabic{section}.\arabic{equation}}
\newcommand{\newsection}[1]{\setcounter{equation}{0} \section{#1}}
\title{Characterisation of conditional weak mixing via ergodicity of the tensor product in Riesz Spaces\footnote{
Funded in part by the NRF joint South Africa - Tunisia gant SATN180717350298, grant number 120112.
}}
\author{
 Mohamed Amine Ben Amor\\
 Research Laboratory of Algebra, Topology, Arithmetic, and Order\\
 Tunis-El Manar University, 2092-El Manar, Tunisia\\ \\
 Jonathan Homann\\
 Department of Mathematics,\\  Research Focus Area: Pure and Applied Analytics\\
 North-WestUniversity\\
 Potchefstroom, 2531, South Africa\\ \\
  Wen-Chi Kuo \&
 Bruce A. Watson\footnote{{Supported in part by the Centre for Applicable Analysis and
Number Theory.}}
\footnote{e-mail: b.alastair.watson@gmail.com} \\
 School of Mathematics\\
 University of the Witwatersrand\\
 Private Bag 3, P O WITS 2050, South Africa}

\maketitle
\abstract{We link conditional weak mixing and ergodicity of the tensor product in Riesz spaces. 
In particular, we characterise conditional weak mixing of a conditional expectation preserving system by the ergodicity of its tensor product with itself or other ergodic systems. In order to achieve this we characterise all band projections in the tensor product of two Dedekind complete Riesz spaces with weak order units.}

\parindent=0cm
\parskip=0.5cm
\section{Introduction} \label{s: introduction}
In Eisner et al \cite[Theorem 9.23]{EFHN}, Krengel, \cite[pg. 98]{krengel} and Petersen, \cite[Section 2.6, Theorem 6.1]{ergodic book}, characterisation of weakly mixing systems via the tensor product of ergodic systems is presented. The concept of ergodicity was discussed in the Riesz space setting in \cite{khw-1-B}.

In the classical setting, given a probability space $\parent{\Omega,\mathcal{A},\mu}$ and a measure preserving transformation $\tau : \Omega \rightarrow \Omega$, the measure preserving system $\parent{\Omega,\mathcal{A},\mu,\tau}$ is ergodic if and only if
	\[
	\lim_{n \rightarrow \infty}
	\frac{1}{n}
	\sum_{k=0}^{n-1}
	\func{\mu}{ \func{{\tau}^{-k}}{A} \cap B}
	=
	\func{\mu}{A} \func{\mu}{B}
	\]
for all $A,B \in \mathcal{A}$.
Here $\parent{\Omega,\mathcal{A},\mu,\tau}$ is said to be weakly mixing if
	\[
	\lim_{n \rightarrow \infty}
	\frac{1}{n}
	\sum_{k=0}^{n-1}
	\abs{
	\func{\mu}{ \func{{\tau}^{-k}}{A} \cap B}
	-
	\func{\mu}{A} \func{\mu}{B}
	}
	=0,
	\]
for all $A,B \in \mathcal{A}$.

Various aspects of mixing processes have been considered in the Riesz space setting in \cite{khw-1-B, KRW, KVW} and other for some other aspects of stochastic processes in Riesz spaces see for example \cite{jensen paper}.
We recall from \cite{expectation paper} that an operator, $T$, on a Dedekind complete Riesz space, $E$, with weak order unit is said to be a conditional expectation operator if $T$ is a positive order continuous projection on $E$ which maps weak order units to weak order units and has range $\range{T}$ a Dedekind complete Riesz subspace of $E$ (i.e. $\range{T}$ is order closed in $E$).  We say that a conditional expectation operator, $T$, is strictly positive if $Tx=0$ with $x \in E_+$ implies $x=0$. The setting in which we work is a Dedekind complete Riesz space, $E$, with weak order unit, say $e$. In this setting every band $B$ in $E$ is a principle projection band and is thus the range of a band projection $P$, Here $Pe=p$ is a component of $e$ which is generator of $B$. We recall that $p$ is a component of $e$ if $0\le p\le e$ and $(e-p)\wedge p=0$. This gives bijections between the bands of $E$, the band projections on $E$ and the components of $e$. In the current work, for notational compactness we will work largely in terms of components. We will denote by $C_E(e)$ the components of $e$ in $E$.

In \cite{khw-1-A} we generalised the notion of a  measure preserving system to the Riesz space setting as follows. 
Let $E$ be a Dedekind complete Riesz space with weak order unit, say, $e$, let $T$ be a conditional expectation operator on $E$ with $Te=e$ and let $S$ be a Riesz homomorphism on $E$ with $Se=e$. If $TSp=Tp$, for all components $p$ of $e$, then $\parent{E,T,S,e}$ is called a conditional expectation preserving system. Due to Freudenthal's Spectral Theorem, \cite[Theorem 33.2]{zaanen}, the condition $TSp=Tp$ for all components $p$ of $e$ in the above is equivalent to $TSf=Tf$ for all $f \in E$.

In \cite{khw-1-B} we obtained that the conditional expectation preserving system $\parent{E,T,S,e}$ is ergodic if and only if
\begin{equation} \label{ergodic}
\frac{1}{n} \sum_{k=0}^{n-1} \func{T}{\parent{S^kp} \cdot q}\to Tp \cdot Tq,
\end{equation}
in order, as $n\to\infty$, for all components $p$ and $q$ of $e$. Further, we defined 
 $\parent{E,T,S,e}$ to be conditionally weak mixing if
\begin{equation} \label{weak mixing}
\frac{1}{n} \sum_{k=0}^{n-1} \left|\func{T}{\parent{S^kp} \cdot q}- Tp \cdot Tq\right|\to 0,
\end{equation}
in order, as $n\to\infty$, for all componenets $p$ and $q$ of $e$.
 Here $\cdot$ denote multiplication in the $f$-algebra 
$E_e= \setbuilder{f \in E}{\abs{f} \leq ke \textrm{ for some $k \in {\2R}_+$}}$, the subspace of $E$ of $e$ bounded elements of $E$, see \cite{azouzi, venter, zaanen2}. 
  The $f$-algebra structure on $E_e$ gives $p \cdot q = p\wedge q$ for all components $p$ and $q$ of $e$.
 Further, if $T$ is a conditional expectation operator on $E$ with $Te=e$, then $T$ is also a conditional expectation operator on $E_e$.
 We refer the reader to Aliprantis and Border \cite{riesz book}, Fremlin \cite{fremlin}, Meyer-Nieberg \cite{MN-BL}, Venter \cite{venter} and Zaanen \cite{zaanen} and \cite{zaanen2} for background on Riesz spaces and $f$-algebras.

The main result of the present paper characterises conditional weak mixing in a Riesz space by ergodicity in the tensor product as follows.

\begin{theorem} \label{product spaces theorem}
In a conditional expectation preserving system $\parent{E,T,S,e}$, with $T$ strictly positive, the following are equivalent.
\begin{enumerate}
\item 
\label{product spaces theorem 1}
$\parent{E,T,S,e}$ is conditionally weak mixing.
\item
\label{product spaces theorem 4}
$\parent{\dedotimes{E}{E_1}, T \otimes_\delta T_1, S \otimes_\delta  S_1,e\otimes e_1}$ is ergodic for each ergodic $\parent{E_1,T_1,S_1,e_1}$ with $T_1$ strictly positive.
\item
\label{product spaces theorem 5}
$\parent{\dedotimes{E}{E},T \otimes_\delta T, S \otimes_\delta  S,e\otimes e}$ is ergodic.
\end{enumerate}
\end{theorem}

This result is in Section 4, aspects of  the on tensor products of Dedekind complete Riesz spaces with weak order units required for this are given in Section 3 while general prel;iminaries are presented in Section 2. 

\section{Preliminaries} \label{preliminaries}
From \cite[Lemma 2.1]{KRodW2} it follows that
if $E$ is a Dedekind complete Riesz space and $\parent{f_n}$ is an order convergent sequence in $E$ with order limit $0$
then
$\displaystyle{\frac{1}{n} \sum_{k=0}^{n-1}|f_k|\to 0}$, in order, as $n\to\infty$. Further, the order convergence of the sequence of partial sums 
$\displaystyle{\left(\sum_{k=0}^{n-1}|f_k|\right)}$ implies the order convergence of the sequence of partial sums
$\displaystyle{\left(\sum_{k=0}^{n-1}f_k\right)}$. Moreover, if $\parent{f_n}$ has order limit $f$, then
$\displaystyle{\frac{1}{n} \sum_{k=0}^{n-1}f_k\to f}$, in order, as $n\to\infty$.

Let $\parent{E,T,S,e}$ be a conditional expectation preserving system. For $f \in E$ and $n \in \2N$, we define
	\begin{subequations}
	\begin{align}
	\label{Sn definition}
	S_nf
	&
	:=
	\frac{1}{n}
	\sum_{k=0}^{n-1} S^k f,
	\\
	\label{L definition}
	L_Sf
	&
	:=
	\lim_{n \to \infty}
	S_nf
	,
	\end{align}
	\end{subequations}
where the above limits are order limits, if they exist.

We say that $f \in E$ is $S$-invariant if $Sf=f$. The set of all $S$-invariant $f \in E$ will be denoted $\3I_S:= \setbuilder{f \in E}{Sf=f}.$ The set of $f \in E$ for which $L_Sf$ exists will be denoted $\3E_S$.
As shown in \cite{khw-1-B},  $\3I_S \subset \3E_S$ and $L_Sf=f$, for all $f \in \3I_S$.
Using this, Birkhoff's (Bounded) Ergodic Theorem, \cite[Theorem 3.7]{ergodic paper}, gives that, in a conditional expectation preserving system, $\parent{E,T,S,e}$, for each $f \in E$, the sequence $\parent{S_nf}$ is order bounded in $E$ if and only if $f \in {\mathcal{E}}_S$ and, further to this, for each $f \in {\mathcal{E}}_S$, $L_Sf=SL_Sf$ and $TL_Sf=Tf$. Thus, $L_S: \3E_S \to \3I_S$.
Finally, if $E= {\mathcal{E}}_S$ then $L_S$ is a conditional expectation operator on $E$ with $L_Se=e$. By restricting out attention to $E_e$, we have that $\parent{S_nf}$ is order bounded in $E_e$ for each $f \in E_e$ and so, by Birkhoff's (Bounded) Ergodic Theorem, $E_e \subset {\mathcal{E}}_S$, moreover, $L_Sf \in E_e$, giving that $L_S|_{E_e}$ is a conditional expectation operator on $E_e$.

As in \cite{khw-1-B} we say that the conditional expectation preserving system $\parent{E,T,S,e}$ is ergodic if $L_Sf \in \range{T}$ for all $f \in \3I_S$. As shown in  \cite{khw-1-B}, $\parent{E,T,S,e}$ is ergodic if and only $L_Sf=Tf$ for all $f \in \3I_S$. We now recall some results from \cite{khw-1-B}.
\begin{corollary} \label{tsm corollary}
The conditional expectation preserving system $\parent{E,T,S,e}$ with $E=\3E_S$ is ergodic if and only if $T=L_S$.
\end{corollary}
	
\begin{corollary} \label{corollary for mixing}
The conditional expectation preserving system $\parent{E,T,S,e}$ is ergodic if and only if
\begin{equation} \label{weak mixing section first equation}
\frac{1}{n} \sum_{k=0}^{n-1} \func{T}{\parent{S^kp} \cdot q}\to Tp \cdot Tq,
\end{equation}
in order as $n\to\infty$, for all components $p$ and $q$ of $e$.
\end{corollary}

In \cite{ergodic book}, a subset $N$ of $\N_0$ is said to be of density zero if
	\[
	\frac{1}{n}
	\sum_{k=0}^{n-1}
	\func{{\chi}_N}{k}
	\to
	0
	\]
as $n \to \infty$, where $\displaystyle \func{{\chi}_N}{k}=0$ if $k \in \2N_0 \setminus N$ and $\displaystyle \func{{\chi}_N}{k}=1$ if $k \in N$. 

The Koopman-von Neumann Lemma \cite[Lemma 6.2]{ergodic book} asserts that if a sequence $(f_n)$ of real numbers is non-negative and bounded, then $\displaystyle \frac{1}{n} \sum_{k=0}^{n-1} f_k \to 0$ as $n \to \infty$ if and only if there is $N$, a subset of $\2N_0$ of density zero, such that $f_n \to 0$ as $n \to \infty$, for $n \in \2N_0 \setminus N$.
This was extended to the Riesz space context in \cite{khw-1-A} as follows.

\begin{definition}\label{subset of density zero definition}
Let $E$ be a Dedekind complete Riesz space with weak order unit $e$.
A sequence $\displaystyle {\parent{p_n}}$ of components of $e$ is said to be of density zero if
	\[
	\frac{1}{n}
	\sum_{k=0}^{n-1}
	p_k
	\to 0,
	\]
	in order, as $n\to\infty$.
\end{definition}

If $E$ is a Dedekind complete Riesz space with weak order unit $e$ and $E^u$ is the universal completion of $E$ then $E^u$ and the ideal generated by $e$, $E_e$ are $f$-algebras. With the multiplicative structure of $E^u$ we observe that $E$ is an $E_e$-module, i.e. elements of $E$ can be multiplied by elements of $E_e$ to give back an element of $E$. In particular if $P$ is a band projection on $E$ and $p=Pe$ then $pf=Pf$ for all $f\in E$. The Koopman-von Neumann Lemma in Riesz spaces of \cite{khw-1-A} can now be written as follows.

\begin{theorem}[Koopman-von Neumann] \label{koopman-von neumann lemma}
Let $E$ be a Dedekind complete Riesz space with weak order unit and let
 $(f_n)$ be an order bounded sequence in the positive cone $E_+$, of $E$, then
$\displaystyle{\frac{1}{n} \sum_{k=0}^{n-1} f_k\to 0},$ 
in order, as $n\to\infty$, if and only if there exists a density zero sequence $\parent{p_n}$ of components of $e$ such that 
$\parent{e-p_n} f_n \to 0,$ in order, as $n\to\infty$.
\end{theorem}

We recall the following characterisation theorem from \cite{khw-1-A}.

\begin{theorem} \label{weak mixing big theorem}
Given the conditional expectation preserving system $\parent{E,T,S,e}$, then the following statements are equivalent.
\begin{enumerate}
\item
\label{weak mixing big theorem 1}
$\parent{E,T,S,e}$ is conditionally weak mixing.
\item
\label{weak mixing big theorem 2}
\[
\frac{1}{n} \sum_{k=0}^{n-1} \abs{\func{T}{\parent{S^kf} \cdot g}-Tf \cdot Tg} \rightarrow 0,
\]
in order, as $n \rightarrow \infty$, for all $f,g \in E_e$.
\item
\label{weak mixing big theorem 3}
Given components $p$ and $q$ of $e$, there is a sequence of density zero components, $\displaystyle {\parent{r_n}},$ of $e$, such that
\[
\parent{e-r_n}
\abs{\func{T}{\parent{S^np}\cdot q}-Tp \cdot Tq}
\rightarrow 0,
\]
in order, as $n \rightarrow \infty$.
\end{enumerate}
\end{theorem}

We recall the following theorem from \cite{khw-1-B}.
	
\begin{theorem} \label{weak implies ergodic}
If a  conditional expectation preserving system is conditionally weak mixing then it is ergodic.
\end{theorem}

\section{Components in the tensor product}

For the general construction of the tensor product of Riesz spaces and $f$-algebras we refer the reader to \cite{azouzi2, buskes, fremlin tensor paper, labuschagne tensor paper, van gaans, van waaij}.
 As shown in \cite{labuschagne tensor paper}, if $E$ and $F$ are Archimedean Riesz spaces then a partial ordering can be induced on $E \otimes F$ by the cone generated by $E_+ \otimes F_+$. Further, from the multilinearity of the tensor product, we have that if $0 \geq f \in E$ and $0 \geq g \in F$ then $f \otimes g \geq 0$ in $E \otimes F$. 

Let $E$ and $F$ be a Dedekind complete Riesz spaces with weak order units $e$ and $f$ respectively.
Denote by $\overline{E \overline{\otimes} F}^{\delta}$ the Dedekind completion of the Fremlin tensor product, $E \overline{\otimes} F$, of $E$ and $F$.
In the recent work of J.J. Gobler \cite{grobler} on the Dedekind completion of the Fremlin tensor product of Dedekind complete Riesz spaces with weak order units it was shown that the tensor product of conditional expectation operators yields a conditional expectation operator. 
Here if $e$ and $f$ are a weak order units of $E$ and $F$ respectively then $e \otimes f$ is a weak order unit of $E \overline{\otimes} F$ which implies $e \otimes f$ is a weak order unit of $\dedotimes{E}{F}$
as $E \overline{\otimes} F$ is order dense in  $\dedotimes{E}{F}$.
In particular if $T_E$ and $T_F$ are conditional expectation operators on $E$ and $F$ respectively with $T_Ee=e$ and $T_Ff=f$ then $T_E\otimes_\delta T_F$ is a conditional expectation operator on $\dedotimes{E}{F}$ with $\dedotimes{E}{F}(e\otimes f)=e\otimes f$ and if $T_E$ and $T_F$ are strictly positive then so is $T_E\otimes_\delta T_F$. Further Grobler showed that the map $(f,g)\mapsto f\otimes_\delta g$ from $E\times F \to \dedotimes{E}{F}$ is order continuous. 

\begin{theorem}\label{density-components}
Let $E$ and $F$ be a Dedekind complete Riesz spaces with weak order units $e$ and $f$ respectively.
\begin{itemize}
\item[(a)]
If $p\in C_E(e)$ and $q\in C_F(f)$ then $p\otimes q\in C_{\dedotimes{E}{F}}(e\otimes f)$. 
\item[(b)]
If $u\in C_{\dedotimes{E}{F}}(e\otimes f)$ then
 $$u=\sup\{p\otimes q\,|\, p\otimes q, p\in C_E(e), q\in C_F(f)\}.$$
\end{itemize}
 \end{theorem}
 
 \proof 
 {\bf (a)} If $p\in C_E(e)$ and $q\in C_F(f)$ then $0\le p\otimes q\le e\otimes f$ and 
 $$e\otimes f -p\otimes q = (e-p)\otimes f+p\otimes (f-q)=((e-p)\otimes f)\bigvee(p\otimes (f-q))$$
 since $$0\le ((e-p)\otimes f)\bigwedge(p\otimes (f-q))\le ((e-p)\otimes f)\bigwedge(p\otimes f)=0$$ 
as $(e-p)\wedge p=0$.
 Now 
 \begin{eqnarray*}
 (e\otimes f -p\otimes q)\wedge(p\otimes q)&=&
 (((e-p)\otimes f)\bigvee(p\otimes (f-q)))\bigwedge (p\otimes q)\\
 &=&
(((e-p)\otimes f)\wedge (p\otimes q)) \bigvee ((p\otimes (f-q))\wedge(p\otimes q)).
\end{eqnarray*} 
Here
$$0\le ((e-p)\otimes f)\wedge (p\otimes q))\le ((e-p)\otimes f)\wedge (p\otimes f))=((e-p)\wedge p)\otimes f=0$$
and
$$((p\otimes (f-q))\wedge(p\otimes q))=p\otimes ((f-q)\wedge q)=0.$$  
Thus $p\otimes q$ is a component of $e\otimes f$.

{\bf (b)}
Let $u\in C_{\dedotimes{E}{F}}(e \otimes f)$, then $0 \leq u \leq e \otimes f \in \dedotimes{E}{F}$ and $u \wedge \parent{(e \otimes f )- u}=0$.
Now $E_+ \otimes F_+$ is order dense in  $E_+ \overline{\otimes} F_+$ which is order dense in $\dedotimes{E}{F}$, so
\[u=\sup\setbuilder{w \in E_+ \otimes F_+}{w \leq u}.\]

Let 
\[v=\sup\setbuilder{p \otimes q}{p \otimes q \leq u, p\in C_E(e), q\in C_F(f)},\]
where this supremum is taken in $\dedotimes{E}{F}$.
Here $v$ is a supremum of components of $u$ and is thus a component of $u$, giving that 
$0\le v\le u$ and $(u-v)\wedge v=0$. Further as $u$ is a component of $e\otimes f$, so is $v$.
It remains to show that $u=v$. Assume $v\ne u$, then $0 < u-v \in \dedotimes{E}{F}$.
By the order density of $E\overline{\otimes}F$ in $\dedotimes{E}{F}$, there is $0<h\le u-v$ with $h\in E\overline{\otimes}F$. From \cite{labuschagne tensor paper} there are $x\in E_+$, $y\in F_+$ with $x\ne 0$ and $y\ne 0$ such that $x\otimes y\le h$.
Let $p$ and $q$ denote the components of $e$ and $f$ respectively which generate the bands in $E$ and $F$ generated by $x$ and $y$ respectively.

Let $B_x(E)$ denote the band in $E$ generated by $x$, then
 $B_x(E)=B_p(E)$ and similarly $B_y(F)=B_q(F)$. From \cite{amine},
\[ p\otimes q\in
{B_p(E)}\overline{\otimes}{B_q(F)}={B_x(E)}\overline{\otimes}{B_y(F)}
\subset
\dedclosure{B_{x\otimes y}(E\overline{\otimes} F)}
\subset
B_{u-v}(\dedotimes{E}{F}).
\]
Thus $p\otimes q\in B_{u-v}(\dedotimes{E}{F})$ and is a component of $e\otimes f$, but $u-v$ is also a component of  $e\otimes f$, giving
$$0<p\otimes q\le u-v.$$
Hence $(p\otimes q)\wedge v=0$ making $$v<v \vee (p\otimes q)\le u$$
which contradicts the definition of $v$, so $u=v$.
\qed

\begin{proposition}\label{ceps-tensor}
If $(E_i,T_i,S_i,e_i), i=1,2,$ are conditional epxectation preserving system then
 $(\dedotimes{E_1}{E_2},T_1\otimes_\delta T_2, S_1\otimes_\delta S_2, e_1\otimes e_2)$
 is a conditional expectation preserving system.
 \end{proposition}
 
 \proof
 From \cite{grobler}, we have that 
 $\dedotimes{E_1}{E_2}$ is a Dedekind complete Riesz space, that $e_1\otimes e_2$ is a weak order unit for
 $\dedotimes{E_1}{E_2}$ and that $T_1\otimes_\delta T_2$ is a conditional expectation on $\dedotimes{E_1}{E_2}$
 which is an extension of the map $(T_1\otimes T_2)(f\otimes g):=(T_1f)\otimes (T_2g)$, thus
 $(T_1\otimes_\delta  T_2)(e_1\otimes e_2)=e_1\otimes e_2$. That $S_1\otimes_\delta S_2$ is a Riesz homomorphism on 
 $\dedotimes{E_1}{E_2}$ follows directly, as does the property $(S_1\otimes_\delta  S_2)(e_1\otimes e_2)=e_1\otimes e_2$.
 Finally, 
 \begin{align*}
 (T_1\otimes_\delta T_2)(S_1\otimes_\delta S_2)(p_1\otimes p_2)&=(T_1S_1p_1)\otimes (T_2S_2p_2)\\ &=(T_1p_1)\otimes (T_2p_2)=(T_1\otimes_\delta T_2)(p_1\otimes p_2)
 \end{align*}
 for all $p_i\in C_{e_i}(E_i), i=1,2.$ By Proposition \ref{density-components} $\{p_1\otimes p_2\,|\,p_i\in C_{e_i}(E_i), i=1,2\}$ is order dense in $C_{e_1\otimes e_2}(\dedotimes{E_1}{E_2})$, also $T_1\otimes_\delta T_2$
 and $(T_1\otimes_\delta T_2)(S_1\otimes_\delta S_2)$ are order continuous on $\dedotimes{E_1}{E_2}$, giving
 that $(T_1\otimes_\delta T_2)(S_1\otimes_\delta S_2)=T_1\otimes_\delta T_2$ on $C_{e_1\otimes e_2}(\dedotimes{E_1}{E_2})$ and thus on $\dedotimes{E_1}{E_2}$.
 \qed

\begin{proposition}
Let $E$ and $F$ be Dedekind complete Riesz spaces with weak order units $e$ and $f$ respectively.
Let ${\parent{p_n}}$ and ${\parent{q_n}}$ be a sequence in $C_e(E)$ and $C_f(F)$ respectively, with at least one of them  density zero, then $(p_n\otimes q_n)$ is a density zero sequence components of  $e\otimes f$ in $\dedotimes{E}{ F}$.
\end{proposition}

\begin{proof} 
We assume ${\parent{p_n}}\subset C_e(E)$ and ${\parent{q_n}}\subset C_f(F)$ with ${\parent{p_n}}$ a density zero sequence of components of $e$.
By assumption, $$\frac{1}{n} \sum_{k=0}^{n-1}p_ke \rightarrow 0,$$
 in order, as $n \rightarrow \infty$ and 
 $$0\le \frac{1}{n} \sum_{k=0}^{n-1}q_k  \le f,$$ for all $n\in\N$. 
 Let $r_n:=p_n \otimes q_n$ for each $n \in {\mathbb{N}}_0$, then $(r_n)\subset C_{e\otimes f}(\dedotimes{E}{ F})$
 with $0\leq r_n= p_n \otimes q_n \leq p_n \otimes f$, so
	\begin{equation*}
	0\le \frac{1}{n} \sum_{k=0}^{n-1} r_k
	 \leq
	\frac{1}{n} \sum_{k=0}^{n-1} \parent{p_k \otimes f}
	=	\parent{\frac{1}{n} \sum_{k=0}^{n-1}p_k} \otimes f
	 \rightarrow 0 \otimes f=0,
	\end{equation*}
in order, as $n \rightarrow \infty$. Here we have used the order continuity of the tensor product in $\dedotimes{E}{ F}$, see \cite{grobler}, as well as that, as $\dedotimes{E}{F}$ is the Dedekind completion of the Archimedean Riesz space $E\overline{\otimes} F$,  if a net $(G_\alpha)$ in $E\overline{\otimes} F$ converges in $\dedotimes{E}{F}$ to $G\in E\overline{\otimes} F$ then $(G_\alpha)$ converges in $E\overline{\otimes} F$ to $G$, see \cite[Theorem 32.2]{L-Z}.  
	\qed
\end{proof}


\section{Ergodicity and mixing in tensor product space}

To access the required properties of the tensor product we work via bilinear maps. The following proposition follows directly from the 
properties of multiplication in the $f$-algebra $E_e$.

\begin{proposition}
\label{J proposition}
Let $E$ be a Dedekind complete Riesz space with weak order unit, $e$, where $e$ is also taken to be the algebraic unit of the $f$-algebra $E_e$.
Let $J: E \times E \rightarrow E$ be defined by $\func{J}{f,g}=f \cdot g$ for all $f,g \in E_e$, then $J$ is bilinear, positive and order continuous on $E_e$.
\end{proposition}

\begin{remark}
\label{remark +}
Working in the category of Dedekind complete Riesz spaces, see \cite{grobler}, let $E$ be a Dedekind complete Riesz space with weak order unit $e$ and let $J$ be as in Proposition {\ref{J proposition}}, then the induced map $j: \dedotimes{E_e}{ E_e} \rightarrow E_e$ with additivity defined by $\func{j}{f \otimes g + p \otimes q}=J(f,g)+J(p,q)=fg+pq$ is linear and order continuous.
\end{remark}

Let $\parent{E_i,T_i,S_i,e_i}, i=1,2,$ be conditional expectation preserving systems. If their Dedekind complete tensor product, $(\dedotimes{E_1}{E_2},T_1\otimes_\delta T_2, S_1\otimes_\delta S_2, e_1\otimes e_2)$, is ergodic, then by \eqref{weak mixing section first equation} of Corollary \ref{corollary for mixing} we have that
	\begin{equation}
	\label{tensor ergodic equation}
	\frac{1}{n}\sum_{k=0}^{n-1} T_1((S_1^kp_1)\cdot q_1)\otimes T_2((S_2^kp_2)\cdot q_2) \to  (T_1p_1\cdot T_1q_1)\otimes (T_2p_2\cdot T_2q_2),
	\end{equation}
in order in $\dedotimes{E_1}{E_2}$, as $n \to \infty$,  for all $p_i,q_i\in C_{e_i}(E_i)$, $i=1,2$.
Similarly if this tensor product is conditionally weak mixing, from Theorem \ref{weak mixing big theorem}, we have that 
	\begin{equation}
	\label{tensor weak mixing equation}
	\frac{1}{n}\sum_{k=0}^{n-1}
	\abs{
	T_1((S_1^kp_1)\cdot q_1)\otimes T_2((S_2^kp_2)\cdot q_2) (T_1p_1\cdot T_1q_1)\otimes (T_2p_2\cdot T_2q_2)}
	\to 	0,
	\end{equation}
in order in $\dedotimes{E_1}{E_2}$, as $n \to \infty$,  for all $p_i,q_i\in C_{e_i}(E_i)$, $i=1,2$.

\begin{lemma}\label{tensor-ergodic}
Let $\parent{E,T,S,e}$ and $\parent{E_1,T_1,S_1,e_1}$, with $T$ and $T_1$ strictly positive be 
conditional expectation preserving systems.
\begin{itemize}
\item[(a)] If $\parent{\dedotimes{E}{E_1}, T \otimes_\delta T_1, S \otimes_\delta S_1, e \otimes e_1}$ is ergodic  then
$\parent{E,T,S,e}$ and $\parent{E_1,T_1,S_1,e_1}$ are ergodic.
\item[(b)] If $\parent{\dedotimes{E}{E_1}, T \otimes_\delta T_1, S \otimes_\delta S_1, e \otimes e_1}$ is  conditionally weak mixing then
$\parent{E,T,S,e}$ and $\parent{E_1,T_1,S_1,e_1}$ are conditionally weak mixing.
\end{itemize}
\end{lemma}

\begin{proof}
{\bf (a)} Let $p,q\in C_e(E)$ and $p_1=p, \, p_1=p, \, p_2=e=q_2$ in \eqref{tensor ergodic equation}. By Corollary \ref{corollary for mixing},
\begin{equation*} 
\frac{1}{n}\sum_{k=0}^{n-1} T((S^kp)\cdot q)\otimes e \to  (Tp\cdot Tq)\otimes e,
\end{equation*}
in order as $n\to\infty$, and hence
\begin{equation*} 
\frac{1}{n}\sum_{k=0}^{n-1} T((S^kp)\cdot q) \to  Tp\cdot Tq,
\end{equation*}
in order as $n\to\infty$, showing that $\parent{E,T,S,e}$ is ergodic.
To prove that $\parent{E_1,T_1,S_1,e_1}$ ergodic is similar.

{\bf (b)} Let $p,q\in C_e(E)$ and $p_1=p, \, p_1=p, \, p_2=e=q_2$ in \eqref{tensor weak mixing equation}, giving
\begin{equation*} 
\frac{1}{n}\sum_{k=0}^{n-1}| T((S^kp)\cdot q)\otimes e -  (Tp\cdot Tq)\otimes e|\to 0,
\end{equation*}
in order as $n\to\infty$. Thus
\begin{equation*} 
\frac{1}{n}\sum_{k=0}^{n-1}| T((S^kp)\cdot q) -  Tp\cdot Tq|\to 0,
\end{equation*}
in order as $n\to\infty$, showing that $\parent{E,T,S,e}$ is conditionally weak mixing. To prove that $\parent{E_1,T_1,S_1,e_1}$ is weak mixing is similar. 
\qed
\end{proof}

We are now in a position to prove the main result of the paper, Theorem \ref{product spaces theorem}.

\begin{proof}{\bf (of Theorem \ref{product spaces theorem})}\\
{\bf ({\ref{product spaces theorem 1}})$\Rightarrow$({\ref{product spaces theorem 4}}):}\\
We assume that $\parent{E,T,S,e}$ is conditionally weak mixing and $\parent{E_1,T_1,S_1,e_1}$ is ergodic.
Let $p.q\in C_e(E)$ and $p_1, q_1\in C_{e_1}(E_1)$. Now,
\begin{eqnarray*}
 K_n&:=&\frac{1}{n}\sum_{j=0}^{n-1} (T\otimes_\delta T_1)((S\otimes_\delta S_1)^k(p\otimes p_1)\cdot (q\otimes q_1))\\
 &=&\frac{1}{n}\sum_{j=0}^{n-1} (T\otimes_\delta T_1)((S^kp\otimes S_1^kp_1)\cdot (q\otimes q_1))\\
 &=&\frac{1}{n}\sum_{j=0}^{n-1} (T\otimes_\delta T_1)((S^kp\cdot q)\otimes (S_1^kp_1\cdot q_1))\\
 &=&\frac{1}{n}\sum_{j=0}^{n-1} T(S^kp\cdot q)\otimes T_1(S_1^kp_1\cdot q_1 )\\
 &=&\frac{1}{n}\sum_{j=0}^{n-1} \alpha_k\otimes \beta_k,
 \end{eqnarray*}
where
$\alpha_k:=T((S^kp)\cdot q)$, $\alpha:=Tp\cdot Tq$,
$\beta_k:=T_1((S_1^kp_1)\cdot q_1)$ and 
$\beta:=T_1p_1\cdot T_1q_1$.
As $\parent{E_1,T_1,S_1,e_1}$ is ergodic,
 $$\frac{1}{n}\sum_{j=0}^{n-1}\beta_k -\beta\to 0$$
 in order in $E_1$ as $n\to\infty$,  and, since $\parent{E,T,S,e}$ is conditionally weak mixing,
$$\frac{1}{n}\sum_{j=0}^{n-1}|\alpha_k-\alpha|\to 0$$ 
in order in $E$ as $n\to\infty$.
Hence,
\begin{eqnarray*}
|K_n-\alpha\otimes\beta|&\le&\alpha\otimes\left|\frac{1}{n}\sum_{j=0}^{n-1}\beta_k -\beta\right|+\left(\frac{1}{n}\sum_{j=0}^{n-1}|\alpha_k-\alpha|\right)\otimes \beta_k\to 0,\end{eqnarray*}
in order in $\dedotimes{E}{E_1}$, as $n\to\infty$, giving that
 $K_n\to \alpha\otimes\beta$ in order  in $\dedotimes{E}{E_1}$, as $n\to \infty$.
 Now working in $(\dedotimes{E}{E_1})_{e\otimes e_1}$, i.e. the ideal of $\dedotimes{E}{E_1})$ generated by $e\otimes e_1$, and using \eqref{Sn definition}-\eqref{L definition}, we have that 
 $$\func{L_{S \otimes_\delta  S_1}}{p \otimes p_1}= \lim_{n \to \infty} \frac{1}{n} \sum_{k=1}^{n-1} {\parent{S \otimes_\delta  S_1}}^k \parent{p\otimes p_1},$$
  where the limit is the order limit. Thus
 $$(T\otimes_\delta T_1)(\func{L_{S \otimes_\delta  S_1}}{p \otimes p_1}\cdot (q\otimes q_1)) =(Tp\cdot Tq )\otimes (T_1p_1\cdot T_1q_1).$$
 Here
 $$(Tp\cdot Tq )\otimes (T_1p_1\cdot T_1q_1)=((Tp)\otimes (T_1p_1))\cdot((Tq )\otimes (T_1q_1))$$
 and
 $$((Tp)\otimes (T_1p_1))\cdot((Tq )\otimes (T_1q_1))=((T\otimes_\delta T_1) (p\otimes p_1 ))\cdot((T\otimes_\delta T_1)  (q \otimes q_1)).$$
 The averaging property of conditional expectation operators, see \cite{expectation paper}, can now be applied to give
 $$((T\otimes_\delta T_1) (p\otimes p_1 ))\cdot((T\otimes_\delta T_1)  (q \otimes q_1))=(T\otimes_\delta T_1)((T\otimes_\delta T_1) (p\otimes p_1 ))\cdot(q \otimes q_1)).$$
 Combining the above gives
 $$(T\otimes_\delta T_1)((\func{L_{S \otimes_\delta  S_1}}{p \otimes p_1}-(T\otimes_\delta T_1) (p\otimes p_1 )))\cdot(q \otimes q_1)) =0.$$
 The order continuity and strict positivity of $T\otimes_\delta T_1$ along with the order density,  by Proposition \ref{density-components}, of 
  $\{p_1\otimes p_2\,|\,p_i\in C_{e_i}(E_i), i=1,2\}$  in $C_{e_1\otimes e_2}(\dedotimes{E_1}{E_2})$ applied to the
  above equiality now give
  $$\func{L_{S \otimes_\delta  S_1}}{p \otimes p_1}=(T\otimes_\delta T_1) (p\otimes p_1 ).$$
Again applying the  order continuity of $T\otimes_\delta T_1$ and $L_{S \otimes S_1}$ with the order density of 
  $\{p_1\otimes p_2\,|\,p_i\in C_{e_i}(E_i), i=1,2\}$  in $C_{e_1\otimes e_2}(\dedotimes{E_1}{E_2})$
we have that 
 $L_{S \otimes_\delta  S_1}=T\otimes_\delta T_1$ on the ideal $(\dedotimes{E}{E_1})_{e\otimes e_1}$ and thus on $\dedotimes{E}{E_1}$ making $(\dedotimes{E}{E_1},T\otimes_\delta T_1,S\otimes_\delta S_1, e\otimes e_1)$ ergodic by Corollary~\ref{corollary for mixing}.

{\bf ({\ref{product spaces theorem 4}})$\Rightarrow$({\ref{product spaces theorem 5}}):}\\
By Lemma \ref{tensor-ergodic} $\parent{E,T,S,e}$ is ergodic, so we can take $\parent{E_1,T_1,S_1,e_1}=\parent{E,T,S,e}$.

{\bf ({\ref{product spaces theorem 5}})$\Rightarrow$({\ref{product spaces theorem 1}}):}\\
Assume that $\parent{\dedotimes{E}{E},T \otimes_\delta T, S \otimes_\delta S,e\otimes e}$ is ergodic.
Let $p,q\in C_e(E)$. By Lemma \ref{tensor-ergodic}, $\parent{E,T,S,e}$ is ergodic.
Denote $\alpha_k:=T((S^kp)\cdot q)$ and $\alpha=Tp\cdot Tq$, then \eqref{weak mixing section first equation} gives 
\begin{equation}\label{4-11} 
\frac{1}{n}\sum_{k=0}^{n-1} \alpha_k \to  \alpha,
\end{equation}
 in order as $n\to\infty$.
Taking $p_1=p=p_2$ and $q_1=q=q_2$ in (\ref{tensor ergodic equation}), assumption ({\ref{product spaces theorem 5}}) of   Theorem \ref{product spaces theorem} gives 
\begin{equation}\label{5-11} 
\frac{1}{n}\sum_{k=0}^{n-1} \alpha_k \otimes \alpha_k \to \alpha\otimes \alpha,
\end{equation}
 in $\dedotimes{E}{E}$, in order as $n\to \infty$.
 Furthermore, setting $f_k:={\alpha}_k-\alpha$,
	\begin{align*}
	&\frac{1}{n} \sum_{k=0}^{n-1} f_k	\otimes	f_k	\\
	&=	\frac{1}{n} \sum_{k=0}^{n-1} {\alpha}_k \otimes {\alpha}_k-\parent{\frac{1}{n} \sum_{k=0}^{n-1} {\alpha}_k	}
	\otimes 	\alpha -	\alpha\otimes\parent{\frac{1}{n} \sum_{k=0}^{n-1} {\alpha}_k	} +\alpha \otimes \alpha \\
	& \rightarrow \alpha \otimes \alpha 	- \alpha \otimes \alpha-\alpha \otimes \alpha+\alpha \otimes \alpha=0,
	\end{align*}
 in $\dedotimes{E}{E}$, in order, as $n\to\infty$.
Let $J$ be as in Proposition \ref{J proposition} and $j$ be as in Remark~\ref{remark +}. Since $\displaystyle \frac{1}{n} \sum_{k=0}^{n-1} f_k \otimes f_k \rightarrow \underline{0}$, in $\dedotimes{E}{E}$, in order, as $n \rightarrow \infty$,
by Remark \ref{remark +} we have,
	\begin{subequations}
	\begin{align*}
	\frac{1}{n} \sum_{k=0}^{n-1} f_k^2
	&=\frac{1}{n} \sum_{k=0}^{n-1} \func{j}{f_k \otimes f_k}=
	\func{j}{\frac{1}{n} \sum_{k=0}^{n-1} f_k \otimes f_k}
	\rightarrow \func{j}{0}=0,
	\end{align*}
	\end{subequations}
in $E_e$ (and thus in $E$), in order, as $n \rightarrow \infty$. 
So, by Theorem~\ref{koopman-von neumann lemma} (the Riesz space extension of the Koopman-von Neumann Lemma)  there exists a density zero sequence of components of $e$ in $E$, say $\parent{r_n}$, such that $\parent{e-r_n} \cdot f_k^2 \rightarrow 0$, in $E$, in order, as $n \rightarrow \infty$. 
Hence, $\parent{e-r_n} |f_k| \rightarrow 0$, in order, as $n \rightarrow \infty$ and by Theorem~\ref{weak mixing big theorem} $(E,T,S,e)$ is conditionally weak mixing.
  \qed
\end{proof}

 \end{document}